\documentclass[12pt]%
{amsart}
\usepackage{verbatim}

\usepackage{amssymb}
\usepackage{enumerate}
\usepackage[active]{srcltx}

\newtheorem{theorem}{Theorem}[section]
\newtheorem{proposition}[theorem]{Proposition}
\newtheorem{lemma}[theorem]{Lemma}
\newtheorem{sublemma}[theorem]{Sublemma}
\newtheorem{claim}[theorem]{Claim}

\newtheorem{case}{Case}
\newtheorem{scase}{Case}[case]

\newtheorem{qtheorem}{Koszmider's Theorem}

\theoremstyle{definition}

\newtheorem{definition}[theorem]{Definition}

\newif\ifdeveloping

\ifdeveloping
\usepackage[notref,notcite]{showkeys}
\fi

\newif\ifcommented

\newcommand{\comm}[1]{}

\ifcommented
\renewcommand{\comm}[1]{
\fbox{\fbox{\begin{minipage}{300pt}#1\end{minipage}}
}}

\fi

\newcommand{\ical}{{\mathcal I}}
\newcommand{\pcal}{{\mathcal P}}

\newcommand{\setm}{\setminus}
\newcommand{\empt}{\emptyset}
\newcommand{\subs}{\subset}

\newcommand{\oo}{{\kappa}}
\newcommand{\oot}{{{\kappa}^+}}
\newcommand{\ooth}{{{\kappa}^{++}}}
\newcommand{\om}{{<{\kappa}}}

\def\<{\left\langle}
\def\>{\right\rangle}
\def\cf{\operatorname{cf}}

\def\br#1;#2;{\bigl[ {#1} \bigr]^ {#2} }

\def\to{\longrightarrow}
\def\rank{\operatorname{rk}}

\newcommand{\newcases}{\setcounter{case}{0}}

\newcommand{\lev}[2]{\operatorname{I}_{#1}(#2)}

\newcommand{\conseq}[2]{\<#1\>_{#2}}

\newcommand{\pib}{\operatorname{\pi}_B}

\newcommand{\incof}[1]{\operatorname{E}({#1})}
\newcommand{\eps}[2]{{\epsilon}^{#1}_{#2}}
\newcommand{\veps}[2]{{\varepsilon}^{#1}_{#2}}
\newcommand{\intpart}[1]{\operatorname{\mathcal E}({#1})}

\newcommand{\nn}[1]{\operatorname{n}({#1})}

\newcommand{\contint}[2]{\operatorname{I}({#1},#2)}

\newcommand{\oorbf}{\operatorname{o}}
\newcommand{\oorb}[1]{\oorbf(#1)}
\newcommand{\orbf}{\operatorname{o^*}}
\newcommand{\dorbf}{\operatorname{f}}

\newcommand{\orb}[1]{\orbf(#1)}
\newcommand{\dorb}[2]{\dorbf\{#1,#2\}}

\newcommand{\eorbf}{\operatorname{\overline{o}}}
\newcommand{\eorb}[1]{\eorbf(#1)}

\newcommand{\blocks}{\mathbb B}
\newcommand{\und}{X}
\newcommand{\block}[1]{B_{#1}}
\newcommand{\ifu}{\operatorname{i}}
\newcommand{\undef}{\text{\rm undef}}

\newcommand{\concat}{\mathop{{}^{\frown}\makebox[-3pt]{}}}
\newcommand{\climit}{{\delta}}

\newcommand{\xn}[1]{x_{{\nu},#1}}
\newcommand{\xm}[1]{x_{{\mu},#1}}

\newcommand{\httr}{\operatorname{ht{}^-}}
\newcommand{\htt}{\operatorname{ht{}}}

\newcommand{\prep}{\preceq_{\nu}}
\newcommand{\preq}{\preceq_{\mu}}
\newcommand{\prer}{\preceq}
\newcommand{\prepq}{\preceq_{{\nu},{\mu}}}

\newcommand{\ap}{A_{\nu}}
\newcommand{\aq}{A_{\mu}}
\newcommand{\ar}{A}

\newcommand{\app}{A_p}
\newcommand{\ipp}{\ifu_p}
\newcommand{\prepp}{\preceq_p}

\newcommand{\ip}{\ifu_{\nu}}
\newcommand{\iq}{\ifu_{\mu}}
\newcommand{\ir}{\ifu}

\newcommand{\xp}[1]{x_{{\nu},#1}}
\newcommand{\xq}[1]{x_{{\mu},#1}}

\newcommand{\pre}[1]{\preceq^{R#1}}

\newcommand{\mc}[1]{\mathcal{#1}}
\newcommand{\mbb}[1]{\mathbb{#1}}

\newcommand{\gbar}{\bar g}

\newcommand{\lt}{{\mathcal L}_{\oo}}
\newcommand{\ltd}{\lt^{\delta}}
\newcommand{\aker}{{A_\vartriangle}}

\newcommand{\precc}{\preceq}
\newcommand{\UU}{\operatorname{U}}
\newcommand{\uup}{\UU_{\precc}}

\begin{document}

\title{A consistency result on long cardinal sequences}

\author[J. C. Martinez]{Juan Carlos Martinez}
\address{Facultat de Matem\`atiques i Inform\`atica \\ Universitat de Barcelona \\ Gran
  Via 585 \\ 08007 Barcelona, Spain}

\email{jcmartinez@ub.edu}
\author[L. Soukup]{
Lajos Soukup }
\thanks{The first author was supported by the Spanish Ministry of Education DGI grant MTM2017-86777-P and by the Catalan DURSI grant 2017SGR270. The second author was supported by NKFIH grants nos. K113047 and K129211.
}

\address{Alfr{\'e}d R{\'e}nyi Institute of Mathematics, Hungarian Academy of Sciences\\Budapest, V. Re\'altanoda u. 13-15, H-1053, Hungary}
\email{soukup@renyi.hu}

\begin{abstract}
For any regular cardinal 
 $\kappa$  and ordinal $\eta<\kappa^{++}$ it is consistent that $2^{\kappa}$ is as large as you wish, and 
every function 
  $f:\eta \to [\kappa,2^{\kappa}]\cap Card$ with  $f(\alpha)=\kappa$ for $cf(\alpha)<\kappa$ is the cardinal sequence of some locally compact  scattered space.
\end{abstract}

\maketitle

\section{Introduction}

If $X$ is a
locally compact,
scattered Hausdorff (in short: LCS) space and $\alpha$ is an ordinal, 
we let  $\lev{\alpha}X$ denote the ${\alpha}^{\rm th}$ Cantor-Bendixson level
of $X$. The cardinal sequence of $X$, $CS(X)$, is the sequence of the 
cardinalities of the infinite Cantor-Bendixson levels of $X$, i.e.
$$\mbox{CS}(X) = \langle |I_{\alpha}(X)| : \alpha < \httr(X)\rangle,$$
where $\httr(X)$, the 
  {\em reduced height} of $X$, is the minimal ordinal
${\beta}$ such that  $\lev {\beta}X$ is finite.
The {\em  height} of $X$, denoted by  $\htt(X)$,  is defined as the minimal ordinal
${\beta}$ such that  $\lev {\beta}X=\empt$. 
Clearly $\httr(X)\le \htt(X)\le \httr(X)+1$.

If $\alpha$ is an ordinal, 
let $\mc C( {\alpha})$ denote the class of all cardinal sequences
of LCS spaces of reduced height  ${\alpha}$ and put
$$\mc C_ {\lambda}({\alpha})=\{s\in \mc C({\alpha}): s(0)={\lambda}\land \forall
{\beta}<{\alpha}\ s({\beta})\ge {\lambda}\}.$$

Let $\conseq \kappa\alpha $ denote the constant $\kappa$-valued sequence of length $\alpha$.

In \cite{JSW} it was shown that the class $\mc C(\alpha)$ is described if the classes $\mc C_\kappa(\beta)$ are characterized
for every infinite cardinal $\kappa$ and ordinal $\beta\le \alpha$. Then, under GCH,  a full description of the classes $\mc C_\kappa(\alpha)$ 
for infinite cardinals $\kappa$  and ordinals $\alpha<\omega_2$ was 
given.

The situation becomes, however, more complicated for $\alpha\ge \omega_2$. 
In \cite{MS2} we gave a consistent full characterization of $\mc C_{\kappa}(\alpha)$ for any uncountable regular cardinals $\kappa$ and ordinals $\alpha<\kappa^{++}$ under GCH.

If $GCH$ fails, much less is known on $\mc C_{\kappa}({\alpha})$    even for ${\alpha}<{\kappa}^{++}$.

In \cite{R1} it was proved that $\<\omega\>_{\omega_1}\concat \<\omega_2\>\in \mc C_{\omega}(\omega_1+1)$ is consistent.

In \cite{KM} a similar result was proved for uncountable cardinals instead of $\omega$:  if $\kappa$ is  a regular cardinal with $\kappa^{<\kappa}=\kappa>{\omega}$ and $2^\kappa=\kappa^+$, then in 
some cardinality preserving generic extension of the ground model we have 
$$\<\kappa\>_{\kappa^+}\concat \<\kappa^{++}\>\in \mc C(\kappa^++1).
$$

In \cite{MS3} we proved that if $\kappa$ and $\lambda$ are regular cardinals, $\kappa^{<\kappa}=\kappa$,  $2^\kappa=\kappa^+$,  and ${\delta}<\kappa^{++}$ with $cf({\delta})=\kappa^+$, then in 
some cardinality preserving generic extension of the ground model we have 
$$\<\kappa\>_{\delta}^\frown \<\lambda\>\in \mc C({\delta}+1).
$$

In this paper we will prove a much stronger result than the above mentioned one.  

\begin{theorem}\label{tm:newmain}
 Assume that $\kappa$ and $\lambda$ are regular cardinals, 
 ${\kappa}^{++}\le {\lambda}$, $\kappa^{<\kappa}=\kappa$, 
 $2^\kappa=\kappa^+$, ${\lambda}^{\kappa^+}={\lambda}$ and ${\delta}<\kappa^{++}$.
 Then, in 
some cardinality preserving generic extension of the ground model, we have
$2^{\kappa}={\lambda}$ and 
 \begin{displaymath}
  \{ f\in {}^{{\delta}}([{\kappa},{\lambda}]\cap Card): f({\alpha})={\kappa} \text{ whenever } cf({\alpha})<{\kappa} \}\subs \mc C_{\kappa}({{\delta}}).
 \end{displaymath}
%
\end{theorem}

\begin{definition}
 Let $\mc C$ be a family of sequences of cardinals.
 We say that an LCS space $X$ is {\em universal for $\mc C$}  iff
 $CS(X)\in \mc C$ and for each $s\in \mc C$ there is an open subspace $Z\subs X$ with $CS(Z)=s$.
\end{definition}

Instead of Theorem \ref{tm:newmain} we prove the following stronger result:

\begin{theorem}\label{tm:mainnew}
 Assume that $\kappa$ and $\lambda$ are regular cardinals, 
 ${\kappa}^{++}\le {\lambda}$,
 $\kappa^{<\kappa}=\kappa$, 
 $2^{\kappa}={\kappa}^+$,  $\lambda^{{\kappa}^+}=\lambda$ and ${\delta}<\kappa^{++}$.
Then,  in some cardinal preserving generic extension, 
we have $2^{\kappa}={\lambda}$ and 
there is an LCS space $X$ which is universal for 
 \begin{displaymath}
  \mc C=\{f\in {}^{\delta}\big([\kappa,\lambda]\cap Card\big ): f(\alpha)=\kappa \text{ whenever } cf(\alpha)<\kappa\}.
 \end{displaymath}

\end{theorem}

\begin{definition}\label{goodLCS} Let $\kappa<\lambda$ be cardinals, $\delta$ be an ordinal, and $A\subs \delta$.
An LCS space $X$ of height ${\delta}$  is called {\em $(\kappa,\lambda,\delta,A)$-good}  iff
 there is an open  subspace $Y\subs X$
such that
\begin{enumerate}[(1)]
\item $CS(Y)=\conseq {\oo}{\delta}$,\smallskip
\item $\lev{{\zeta}}{Y}=\lev{{\zeta}}{X}$, and so  $|\lev{{\zeta}}{X}|={\kappa}$, for ${\zeta}\in {\delta}\setm A$,\smallskip
\item $|\lev{{\zeta}}{X}|={\lambda}$  for ${\zeta}\in  A$,\smallskip
\item for $\zeta \in A$ the set $Z_{\zeta}=\lev {<{\zeta}}Y\cup \lev {\zeta}X$ is an open subspace of $X$
such that \smallskip
\begin{enumerate}[(a)]
 \item $\lev{\xi}{Z_{\zeta}}=\lev{\xi}Y$ for ${\xi}< {\zeta}$,\smallskip
 \item $\lev{{\zeta}}{Z_{\zeta}}=\lev {\zeta}X$.
\end{enumerate}
\end{enumerate}
\end{definition}

Theorem \ref{tm:mainnew} follows immediately from Koszmider's Theorem, 
Theorem \ref{tm:mainthm2} and Proposition \ref{pr:ltd} below.

The following result of Koszmider can be obtained by 
putting together \cite[Fact 32 and Theorem 33]{Ko}:

\begin{definition}[See \cite{Ko1,Ko}]
Assume that ${\kappa}<{\lambda}$ are infinite cardinals. We say that a function $\mc F :
[{\lambda}]^2\longrightarrow {\kappa}^+$ is a ${\kappa}^+$-{\em strongly unbounded
function on} ${\lambda}$ iff for every ordinal $\vartheta < {\kappa}^+$  and for every
 family $\mc A\subset [{\lambda}]^{<{\kappa}}$ of
pairwise disjoint sets with $|\mc A|={\kappa}^+$,
 there are different $a,b\in \mc A$
such that $\mc 
F\{{\alpha},{\beta}\}> \vartheta$ for every ${\alpha}\in a$ and $ {\beta}\in
b$. 
\end{definition}

\begin{qtheorem}
If ${\kappa},{\lambda}$ are infinite cardinals such that ${\kappa}^{++} \leq
{\lambda}$, ${\kappa}^{<{\kappa}}={\kappa}$, $2^{{\kappa}}= {\kappa}^+$ and $
{\lambda}^{\kappa^+}={\lambda}$, then 
in some cardinal preserving generic extension 
${\kappa}^{<{\kappa}}={\kappa}$, ${\lambda}^{\kappa}={\lambda}$ and 
there is  a ${\kappa}^+$- strongly unbounded function on ${\lambda}$.
\end{qtheorem}

For an ordinal ${\delta}<\ooth$
let $$\ltd=\big\{{\alpha}<{\delta}:\cf({\alpha})\in \{\oo,\oot\}\big\}.$$

\begin{theorem}\label{tm:mainthm2}
If  $\oo<{\lambda}$ are    regular cardinals with
${\oo}^{<{\oo}}={\oo}$, ${\lambda}^{\kappa}={\lambda}$,  and there is  a ${\kappa}^+$- strongly unbounded function on ${\lambda}$,
then for each ${\delta}<\ooth$ there is
a $\oo$-complete $\oot$-c.c poset $\pcal$ of cardinality $\lambda$ such that
in $V^\pcal$ we have $2^{\kappa}={\lambda}$ and
there is a $(\kappa,\lambda,\delta,\ltd)$-good space.
\end{theorem}

We will prove Theorem \ref{tm:mainthm2} in Section 
\ref{sc:mainthm2}.

\begin{proposition}\label{pr:ltd}
If ${\kappa}<{\lambda}$ are   regular cardinals
and ${\delta}<{\kappa}^{++}$, then a  $(\kappa,\lambda,\delta,\ltd)$-good space 
is universal for 
 \begin{displaymath}
  \mc C=\{f\in {}^\delta\big([\kappa,\lambda]\cap Card\big ): f(\alpha)=\kappa \text{ whenever } cf(\alpha)<\kappa\}.
 \end{displaymath}
 \end{proposition}

\begin{proof} 
Let $X$ be a $(\kappa,\lambda,\delta,\ltd)$-good space.
Fix $f\in \mc C$.
For ${\zeta}\in \ltd$ pick $T_{\zeta}\in \br {\lev {\zeta}X};f({\zeta});$,
and let
\begin{equation}\notag
Z=Y\cup\bigcup\{T_{\zeta}:{\zeta}\in \ltd\}.
\end{equation}
Since $\lev {<{\zeta}}Y\cup T_{\zeta}$ is an open subspace of $X$ for ${\zeta}\in \ltd$, 
for every ${\alpha}<{\delta}$ we have
\begin{equation*}
\lev {\alpha}Z=\lev {\alpha}Y\cup
\bigcup\{\lev {\alpha}{\lev {<{\zeta}}Y\cup T_{\zeta}}:
{\zeta}\in \ltd \}.
\end{equation*}
Since 
\begin{equation*}
\lev {\alpha}{\lev {<{\zeta}}Y\cup T_{\zeta}}=
\left\{
\begin{array}{ll}
\lev {\alpha}Y&\text{if ${\alpha}<{\zeta}$,}\\
T_{\zeta}&\text{if ${\alpha}={\zeta}$,}\\
\empt&\text{if ${\zeta}<{\alpha}$,}
\end{array}
\right. 
\end{equation*}
we have 
\begin{equation*}
\lev {\alpha}{Z}=
\left\{
\begin{array}{ll}
\lev {\alpha}Y&\text{if ${\alpha}\notin \ltd$,}\\
\lev {\alpha}Y\cup T_{{\alpha}}&\text{if ${\alpha}\in \ltd$.}
\end{array}
\right. 
\end{equation*}
Since $|\lev {\alpha}Y|={\kappa}$ and $|\lev {\alpha}Y\cup T_{\alpha}|=
{\kappa}+f({\alpha})=f({\alpha})$, we have 
$CS(Z)=f$, which was to be proved.
\end{proof}

\section{Proof of theorem \ref{tm:mainthm2}}\label{sc:mainthm2}

\subsection*{Graded posets}

In \cite{KM}, \cite{M}, \cite{R1} and in many other papers, the
existence of an LCS space is proved 
 in such a way that instead of constructing the space
directly,  a certain ``graded poset''  is  produced  which
guaranteed the existence of the wanted  LCS-space. 
From these results, Bagaria, \cite{Ba}, extracted
the notion of $s$-posets and established the formal connection between graded posets and LCS-spaces. 
For technical reasons, we will use a reformulation of  Bagaria's result introduced in  
\cite{S}.

If $\precc$ is an arbitrary partial order on a set $X$ then define the topology
${\tau}_{\precc}$ on $X$ generated by the family 
$\{\UU_{\precc}(x),X\setm \UU_{\precc}(x):x\in X\}$ as a subbase, where
$\uup(x)=\{y\in X:y\precc x\}$.

\begin{proposition}[{\cite[Proposition 2.1]{S}}]\label{pr:back}
Assume that $\<X,\preceq\>$ is a poset,   $\{X_{\alpha}:{\alpha}<{\delta}\}$ is a partition of $X$
and 
$i:\br X;2;\to X\cup\{undef\}$
is a function  satisfying
(\ref{p_e})--(\ref{p_l}) below:
\begin{enumerate}[(a)]
\item \label{p_e} 
if $x\in X_{\alpha}$, $y\in X_{\beta}$ and 
$x\precc y$ then either $x=y$ or ${\alpha}<{\beta}$,
\item \label{p_i}$\forall \{x,y\}\in \br X;2;$ $\big ($
$\forall z\in X$ $(z\precc x\land z\precc y)$ iff
 $z\precc i\{x,y\}$ $\big )$.
\item \label{p_l} if $x\in X_{\alpha}$ and ${\beta}<{\alpha}$
then the set $\{y\in X_{\beta}:y\precc x\}$ is infinite.
\end{enumerate}
 Then $\mc X=\<X,{\tau}_{\precc}\>$ is an LCS space
with $\lev {\alpha}{\mc X}=X_{\alpha}$ for ${\alpha}<{\delta}$.
\end{proposition}

\begin{definition}\label{pr:back2}
 Let $\kappa<\lambda$ be cardinals, $\delta$ be an ordinal, and $A\subs \delta$.
Assume that $\<X,\preceq\>$ is a poset,   $\{X_{\alpha}:{\alpha}<{\delta}\}$ is a partition of $X$
and  
$i:\br X;2;\to X\cup\{undef\}$
is a function  satisfying
conditions (\ref{p_e})--(\ref{p_l}) from  Proposition \ref{pr:back}.

We say that poset $\<X,\preceq\>$ is 
 {\em $(\kappa,\lambda,\delta,A)$-good} iff
there is a set $Y\subs X$ 
 such that: 
\begin{enumerate}[(a)]\setcounter{enumi}{3}
 \item \label{d:yopen} if $x_0\preceq x_1$, then either $x_0=x_1$ or $x_0\in Y$;
 \item \label{e:notinA} $X_{\zeta}\in \br  Y;{\kappa};$ for ${\zeta}\in {\delta}\setm A$;
 \item \label{f:inA}  $|X_{\zeta}|={\lambda}$ and $|X_{\zeta}\cap Y|={\kappa}$ for ${\zeta}\in A$.
\end{enumerate}
 \end{definition}

\begin{proposition}\label{pr:ord2sc}
Let $\kappa<\lambda$ be cardinals, $\delta$ be an ordinal, and $A\subs \delta$.
If $\<X,\preceq \>$ is a $(\kappa,\lambda,\delta,A)$-good poset, then 
$\mc X=\<X,{\tau}_{\precc}\>$ is a $(\kappa,\lambda,\delta,A)$-good space.
\end{proposition}

\begin{proof}[Proof of Proposition \ref{pr:ord2sc}]
 By Proposition \ref{pr:back}, $\mc X=\<X,{\tau}_{\precc}\>$ is an LCS space
with $\lev {\alpha}{\mc X}=X_{\alpha}$ for ${\alpha}<{\delta}$.

By (\ref{d:yopen}), the subspace $Y$ is open, and so $\lev{{\zeta}}{Y}=\lev {\zeta}X\cap Y$.
Thus $|\lev {\zeta}Y|={\kappa}$ by (\ref{e:notinA}) and (\ref{f:inA}). So $CS(Y)=\conseq {\oo}{\delta}$, i.e. \ref{goodLCS}(1) holds.

If ${\zeta}\in {\delta}\setm A$, then  $\lev{\zeta}X\subs Y$ by (\ref{e:notinA}), so       $\lev {\zeta}X=\lev{\zeta}Y$. Thus   \ref{goodLCS}(2) holds.
Moreover $\lev{{\zeta}}{Y}=\lev {\zeta}X\cap Y$.

\ref{goodLCS}(3) follows from (\ref{f:inA}).

Also, 
for ${\zeta}\in A$
(\ref{p_e})  and (\ref{d:yopen}) imply that $\uup(s)\subs Z_{\zeta}$ for $s\in Z_{\zeta}$, and so  $Z_{\zeta}$ is an open subspace of $\mc X$.
Hence $\lev{\xi}{Z_{\zeta}}=\lev {\xi}X\cap Z_{\zeta}=X_{\xi}\cap Z_{\zeta}.$

Thus $\lev{\xi}{Z_{\zeta}}=\lev {\xi}Y$ for ${\xi}<{\zeta}$, and 
$\lev{{\zeta}}{Z_{\zeta}}=X_{\zeta}$. So \ref{goodLCS}(4) also holds.

Thus $\mc X$ is a $(\kappa,\lambda,\delta,A)$-good space.
\end{proof}

So, instead of Theorem \ref{tm:mainthm2}, it is enough to prove Theorem \ref{tm:mainthm3} below. 
\begin{theorem}\label{tm:mainthm3}
If  $\oo<{\lambda}$ are    regular cardinals with
${\oo}^{<{\oo}}={\oo}$, ${\lambda}^{\kappa}={\lambda}$,  and there is  a ${\kappa}^+$- strongly unbounded function on ${\lambda}$,
then for each ${\delta}<\ooth$ there is
a $\oo$-complete $\oot$-c.c poset $\pcal$ of cardinality $\lambda$ such that
in $V^\pcal$ we have $2^{\kappa}={\lambda}$ and 
there is a $(\kappa,\lambda,\delta,\ltd)$-good poset.
\end{theorem}

So, assume that ${\kappa}$, ${\lambda}$ and ${\delta}$ satisfy the hypothesis of Theorem 
\ref{tm:mainthm3}.  In order to construct the required poset $\mc P$, first 
we need to recall some notion from \cite[Section 1]{M}.

\subsection*{Orbits}
If ${\alpha}\le {\beta}$ are ordinals let
\begin{equation*}
[{\alpha},{\beta})=\{{\gamma}:{\alpha}\le {\gamma}<{\beta}\}.
\end{equation*}
We say that $I$ is an {\em ordinal interval} iff there are
ordinals ${\alpha}$ and ${\beta}$ with $I=[{\alpha},{\beta})$.
Write $I^-={\alpha}$ and $ I^+={\beta}$.

If $I=[{\alpha},{\beta})$ is an ordinal interval let $\incof
I=\{\veps I{\nu}:{\nu}<\cf({\beta})\}$ be a cofinal closed subset
of $I$ having order type $\cf ({\beta})$ with ${\alpha} = \veps
I{0}$ and put
\begin{equation*}
\intpart I=\{[\veps I{\nu},\veps I{\nu+1}):{\nu}<\cf{\beta}\}
\end{equation*}
 provided ${\beta}$ is a limit ordinal,
and let $\incof I =\{{\alpha},{\beta}'\}$
 and put
\begin{equation*}
\intpart I=\{[{\alpha},{\beta}'),\{{\beta}'\}\}
  \end{equation*}
provided ${\beta}={\beta}'+1$ is a successor ordinal.

\vspace{1mm} Define $\{\ical_n:n<{\omega}\}$ as follows:
\begin{equation*}
\ical_0=\{[0,{\delta})\} \text{ and }
\ical_{n+1}=\bigcup\{\intpart I:I\in \ical_n\}.
\end{equation*}
Put $\mathbb I=\bigcup\{\ical_n:n<{\omega}\}$. 

Note that $\mathbb
I$ is a {\em cofinal tree of intervals} in the sense defined in \cite{M}.
So, the following conditions are
satisfied:

\begin{enumerate}[(i)]

\item  For every $I,J\in {\mathbb I}$, $I\subset J$ or
$J\subset I$ or $I\cap J = \emptyset$.

\item If $I,J$ are different elements of ${\mathbb I}$ with
$I\subset J$ and $J^+$ is a limit ordinal, \mbox{ then  $I^+ < J^+$ }.

 \item $\ical_n$ partitions $[0,\delta )$ for each $n < \omega$.

\item  $\ical_{n + 1}$ refines $\ical_n$ for each $n <\omega$.

\item  For every $\alpha <{\delta}$ there is an $I\in {\mathbb I}$
such that $I^- = \alpha$.
\end{enumerate}

Then, for each ${\alpha}<{\delta}$ we define
\begin{equation*}
\nn {\alpha}=\min\{n:\exists I\in \ical_n \mbox{ with } \
I^-={\alpha}\},
\end{equation*}
and for each  ${\alpha}<{\delta}$ and $n<{\omega}$ we pick
\begin{equation*}
\contint {\alpha}n\in \ical_n \text{ such that } {\alpha}\in
\contint {\alpha}n.
\end{equation*}

\begin{proposition}
\label{Proposition-2.1} Assume that $\zeta <
\delta$ is a limit ordinal.  Then, there is 
 an interval $$J({\zeta})\in \ical_{n(\zeta)-1}\cup\ical_{n(\zeta)}$$ 
such that $\zeta$ is a limit point of $E(J(\zeta))$. 

If $cf(\zeta) = \oot$, then $J(\zeta)\in \ical_{n(\zeta)}$ and 
$J(\zeta)^+=\zeta$.
\end{proposition}

\begin{proof}
%
%
If there is an $I \in \ical_{\nn {\zeta}}$ with $I^+={\zeta}$
then $J(\zeta)=I$. If there is no
such $I$, then ${\zeta}$ is a limit point of $\incof{\contint
{\zeta}{\nn {\zeta}-1}}$, so $J(\zeta)= \contint {\zeta}{\nn
{\zeta}-1}$.

Assume now that $\cf({\zeta})=\oot$. Then
 ${\zeta}\in \incof{\contint {{\zeta}}{\nn {\zeta}-1}}$,
but $|\incof{\contint {{\zeta}}{\nn {\zeta}-1}}\cap {\zeta}|\le
\oo$, so ${\zeta}$ can not be a limit point of $\incof{\contint
{{\zeta}}{\nn {\zeta}-1}}$. Therefore, it has a predecessor
${\xi}$ in $\incof{\contint {{\zeta}}{\nn {\zeta}-1}}$, i.e
$[{\xi},{\zeta}) \in \ical_{\nn {\zeta}}$, and so
$J(\zeta)=[{\xi},{\zeta})$ and
 $J(\zeta)\in \ical_{n(\zeta)}$.
\end{proof}

\vspace{2mm} If $\cf(J(\zeta)^+)\in \{\oo,\oot\}$, we denote by
$\{\eps {\zeta}{\nu}:{\nu}<\cf({J(\zeta)}^+)\}$  the
increasing enumeration of $\incof {J(\zeta)}$, i.e.  
$\eps {\zeta}{\nu}=\veps {J({\zeta})}{\nu}$  for 
${\nu}<\cf ({J(\zeta)^+})$.

\vspace{2mm} Now if ${\zeta}<{\delta}$, we define the {\em basic
orbit} of $\zeta$ (with respect to $\mathbb I$) as
\begin{equation*}
\oorb {\zeta}=\bigcup\{(\incof{\contint{\zeta}m}\cap {\zeta}):
m<\nn {\zeta}\}.
\end{equation*}

We refer the reader to \cite[Section1]{M} 
for some fundamental facts and examples on basic orbits. 
In particular, we have that $\alpha \in o(\beta)$ implies 
$o(\alpha)\subset o(\beta)$.

\vspace{1mm} If  ${\zeta}\in \ltd$,  
we define the {\em extended orbit} of $\zeta$ by
\begin{equation*}
\eorb {\zeta}= \oorb {\zeta}\cup (\incof {J(\zeta)}\cap {\zeta}).
\end{equation*}
Observe that if
$J({\zeta})\in \ical_{n(\zeta)-1}$ then ${\overline{o}}(\zeta) = {o}(\zeta)$.

\bigskip

The underlying set of our  poset will consist of blocks.
The following set 
$\blocks$ below serves as the index set of our blocks:
\begin{equation*}
\blocks=\{S\}\cup \ltd.
\end{equation*}
Let
\begin{equation*}
\block S={\delta}\times \oo
\end{equation*}
and
\begin{equation*}
\block {{\zeta}}=\{{\zeta}\}\times
       [\kappa,\lambda)
\end{equation*}
for ${\zeta}\in \ltd$.

The underlying set of our poset will be
\begin{equation*}
\und=\bigcup\{\block T:T\in \blocks\}.
\end{equation*}

 To obtain a $(\kappa,\lambda,\delta,\ltd)$-good poset 
we take  $Y=\block S$ and  
\begin{equation}\notag
 X_{\zeta}=\left\{ 
 \begin{array}{ll}
  \{{\zeta}\}\times {\kappa}&\text{if ${\zeta}\in {\delta}\setm \ltd$,}\\\\
  \{{\zeta}\}\times {\lambda}&\text{if ${\zeta}\in \ltd$}.
 \end{array}
 \right.
\end{equation}

Define the functions ${\pi}:\und\to {\delta}$
and  $\rho:\und\to \lambda $ by the formulas 
\begin{equation*}
{\pi} (\<{\alpha},{\nu}\>)={\alpha}
\text{ and }\rho (\<{\alpha},{\nu}\>)={\nu}.
\end{equation*}
Define
\begin{equation*}
\pib:\und\to \blocks  \text{ by the formula } x\in \block {\pib
(x)}.
\end{equation*}

Finally we define the {\em orbits} of the  elements  of $\und$ as
follows:
\begin{equation*}
\orb x =\left\{\begin{array}{ll}
\oorb {\pi(x)}&\text{ for }x\in B_S,\\
\eorb {\pi(x)}
     &\text{ for }x\in X\setm B_S,
\end{array}
\right.
\end{equation*}

Observe that $\orb x \in \br \pi(x);\le \oot;$ and
\begin{equation*}
 \text{$|\orb x|\le \oo$ unless $x\in \block{\xi}$ with $cf(\xi)=\oot$.
}
\end{equation*}

To simplify our notation, we will write $\oorb x =\oorb {{\pi} (x)}$
and $\eorb x =\eorb {{\pi} (x)}$.

\subsection*{Forcing construction}

Let $\Lambda\in \mathbb I$ and $\{x,y\}\in \br \und;2;$. We say
that {\em $\Lambda$ separates $x$ from $y$} if
\begin{equation*}
\text{$\Lambda^- <{\pi} (x)<\Lambda^+< \pi(y)$.} 
\end{equation*}

Let $\mc F:\br \lambda;2;\to \oot$ be a 
$\kappa^+$-strongly unbounded function.

Define 
\begin{equation*}
 \dorbf: \br X;2;\to \br \delta;\le {\kappa};
\end{equation*}
as follows:
\begin{equation*}
\dorb xy=\left\{ 
\begin{array}{ll}
o(x)\cup\big\{\eps {\pi(x)}\zeta:\zeta <\mc F\{\rho(x), \rho(y)\}\big\}&
\text{if $\pi_B(x)=\pi_B(y)\ne S$, }\\
&\text{and $cf(\pi(x))={\kappa}^+$,}\\\\
\orb{x}\cap \orb{y}&\text{otherwise}.
\end{array}
\right. 
\end{equation*}

Observe that 
\begin{equation*}
 |\dorb{x}{y}|\le {\kappa}
\end{equation*}
for all $\{x,y\}\in \br X;2;$.

\begin{definition}

\vspace{2mm} Now, we define the poset $\pcal=\<P,\le\>$ as
follows: $\<A,\preceq,i\>\in P$ iff
\begin{enumerate}[(P1)]
\item $A\in \br \und;\om;$; 
\item $\preceq$ is a partial order on
$A$ such that
$x\preceq y$ implies $x=y$ or ${\pi}(x)<{\pi} (y)$; 
\item if  $x\preceq y$ and 
 $\pib(x)\ne S$, then $x=y$;
\item $\ifu:\br A;2; \to A\cup\{{\undef}\}$ such that
for each $\{x,y\}\in \br A;2;$ we have
  \begin{equation}\notag
  \forall a\in A ([a\preceq x\land a\preceq y]\text{ iff } a\preceq \ifu\{x,y\});
  \end{equation}
\item for each $\{x,y\}\in \br A;2;$ if $x$ and $y$ are
$\preceq$-incomparable but $\preceq$-compatible, then
\begin{equation*}
{\pi}(\ifu\{x,y\})\in \dorb{x}{y}; 
\end{equation*}

\item If $\{x,y\}\in [A]^2$ with $x\prec y$, and  
  $\Lambda\in \mathbb I$ separates $x$
    from $y$, then there is
 $z\in A$ such that $x\prec z\prec y$ and ${\pi}
      (z)={\Lambda}^+$.

\end{enumerate}

The ordering on $P$ is the extension: $\<A,\preceq,\ifu\>\le
\<A',\preceq',\ifu'\>$
 iff $A'\subset A$, $\preceq'=\preceq\cap (A'\times A')$, and
$\ifu'\subs \ifu$.
\end{definition}

For $p\in P$  write $p=\<A_p,\preceq_p,\ifu_p\>$.

To complete the proof of Theorem \ref{tm:mainthm3}  we will use the following lemmas which will be proved later:

\begin{lemma}
\label{Lemma-2.2} $\pcal$ is
$\oo$-complete.
\end{lemma}

\begin{lemma}
\label{Lemma-2.3}  $\pcal$
 satisfies the $\oot$-c.c.
\end{lemma}

\begin{lemma}
\label{Lemma-2.4}  (a) 
For all $x\in X$, the set
\begin{displaymath}
 D_x=\{q\in P: x\in A_q\}
\end{displaymath}
is dense in $\mc P$.

\noindent (b)
If $x\in X$, ${\alpha}<\pi(x)$ and ${\zeta}<{\kappa}$, 
then the set 
\begin{displaymath}
 E_{x,{\alpha},{\zeta}}=\{q\in P: x\in A_q \land \exists b\in  A_q\cap (\{{\alpha}\}\times ({\kappa}\setm {\zeta}))\      b\preceq_q x    \}
\end{displaymath}
is dense in $\mc P$
\end{lemma}

  Since $\lambda^{<\kappa}=\lambda$ ,    the cardinality of $P$ is $\lambda$. 
Thus,  Lemma \ref{Lemma-2.2} and   Lemma \ref{Lemma-2.3}  above guarantee  
that forcing with $P$ preserves cardinals and $2^{\kappa}={\lambda}$ in the generic extension.

Let $G\subs P$ be  a generic filter. 
Put 
$A=\bigcup\{\app:p\in G\}$,
$i=\bigcup\{\ipp:p\in G\}$ and
$\preceq=\bigcup\{\prepp:p\in G\}$.
Then $A=X$ by Lemma \ref {Lemma-2.4}(a).

We claim that $\<X,\preceq\>$ is a $({\kappa},{\lambda},{\delta},\ltd)$-poset.

Recall that we put $X_{\zeta}=\{{\zeta}\}\times {\kappa}$ for ${\zeta}\in {\delta}\setm \ltd$ and
$X_{\zeta}=\{{\zeta}\}\times {\lambda}$ for ${\zeta}\in  \ltd$.
Then the poset $\<X,\preceq\>$, the partition $\{X_{\zeta}:{\zeta}<{\delta}\}$, the function $i$
and $Y={\delta}\times {\kappa}$ clearly satisfy conditions \ref{pr:back}(a,b) and 
\ref{pr:back2}(d,e,f) by the definition of the poset $\mc P$.

Finally condition \ref{pr:back}(c) holds by Lemma \ref {Lemma-2.4}(b).

So to complete the proof of Theorem \ref{tm:mainthm3} we need  to prove Lemmas  \ref{Lemma-2.2}, \ref{Lemma-2.3} and \ref{Lemma-2.4}.

\vspace{2mm} Since ${\kappa}$ is regular,  Lemma \ref{Lemma-2.2} clearly
holds.

\begin{proof}[Proof of Lemma \ref{Lemma-2.4}](a)  
Let $p\in P$ be arbitrary. We can assume that $x\notin A_p$.

Let $A_q=A_p\cup \{x\}$, $\preceq_q=\preceq_p\cup\{\<x,x\>\}$, and define $i'\supset i$ such that $i'\{a,x\}=undef$ for $a\in A_p$. 
Then $q=\<A_q,\preceq_q,\ifu_q\>\in D_x$  and $q\le p$.

(b)  Let $p\in P$ be arbitrary. By (a) we can assume that $x\in A_p$. 
Write ${\beta}={\pi} (x)$. Let $K$ be a finite subset of $
[{\alpha},{\beta})$ such that  ${\alpha}\in K$ and 
${\contint {\gamma}n }^+\in K\cup [{\beta},{\delta})$ for
${\gamma}\in K$ and $n<{\omega}$.  For each ${\gamma}\in K$ pick
$b_{\gamma}\in  (\{{\gamma}\}\times  ({\kappa}\setm {\zeta})) \setm A_p$. So ${\pi}(b_{\gamma})={\gamma}$.

Let $A_q=A_p\cup\{b_{\gamma}:{\gamma}\in K\}$,
\begin{multline}\notag
\preceq_q=\preceq_p\cup\{\<b_{\gamma},b_{{\gamma}'}\>:
{\gamma},{\gamma}'\in K, {\gamma}\le {\gamma}'\}
\\\cup\{\<b_{\gamma},z\>:{\gamma}\in K, z\in A_p, x\preceq_p z\}.
\end{multline}
We let $\ifu_q\{y,z\}=\ifu_p\{y,z\}$ if $\{y,z\}\in \br A_p;2;$,
$\ifu_q\{b_{\gamma},b_{\gamma'}\}=b_{\gamma}$ if ${\gamma},{\gamma'}\in K$
with ${\gamma}<{\gamma}'$, $\ifu_q\{b_{\gamma},z\}=b_{\gamma}$ if 
${\gamma}\in K$ and $x\preceq_p z$, and  $\ifu_q\{b_{\gamma},z\}=undef$ otherwise.
%

Let $q=\<A_q,\preceq_q,\ifu_q\>$.  Next we check $q \in P$. 
Clearly $(P1),(P2), (P3)$ and $(P5)$ hold for $q$.
(P4) also holds because 
if $y\in A_p$
and ${\gamma}\in K$ then either $b_{\gamma}\preceq_q y$ or they are $\preceq_q$-incompatible.
To check (P6) it is enough to observe that if $\Lambda$ separates $b_\gamma$ and $y$, then 
 $z=b_{\Lambda^+}$ meets the requirements of (P6). 

By the construction, $q\le p$. 
 
Finally $q\in E_{x,{\alpha},{\zeta}}$ because $b_{\alpha}\in A_q\cap \big(\{{\alpha}\}\times ({\kappa}\setm {\zeta})\big)$
and $b_{\alpha}\preceq_q x$.
\end{proof}

The rest of the paper is devoted to the proof of  Lemma \ref{Lemma-2.3}.

\begin{proof}[Proof of Lemma \ref{Lemma-2.3}]

Assume that $\<r_{\nu}:{\nu}<\oot\>\subs P$ with $r_{\nu}\neq r_{\mu}$ for $\nu <\mu<\kappa^+$.

In the first part of the proof, till Claim \ref{good-pair}, we will find ${\nu}<{\mu}<{\kappa}^+$ such that
$r_{\nu}$ and $r_{\mu}$ are twins in a strong sense, and 
$r_{\nu}$
and $r_{\mu}$ form a {\em good pair}  (see Definition \ref{Definition 2.8}).  
Then,  in the second part of the proof, we will show that if $\{r_{\nu},r_{\mu}\}$ is a good pair, then $r_{\nu}$
and $r_{\mu}$ are compatible in $\mc P$.

Write $r_{\nu}=\<A_{\nu},\preceq_{\nu},\ifu_{\nu}\>$ and
$A_{\nu}=\{\xn i:i<{\sigma}_{\nu}\}$.

Since we are assuming that $\kappa^{<\kappa}=\kappa$, by thinning
out $\<r_{\nu}:{\nu}<\oot\>$ by means of standard combinatorial
arguments, we can assume the following:
\begin{enumerate}[(A)]
\item ${\sigma}_{\nu}={\sigma}$ for each ${\nu}<\oot$. \item
$\{A_{\nu}:{\nu}<\oot\}$ forms a $\Delta$-system with kernel $\aker$.
\item For each ${\nu}<{\mu}<\oot$ there is an isomorphism
$h_{{\nu},{\mu}}:\<A_{\nu},\preceq_{\nu}, \ifu_{\nu}\>\to
\<A_{\mu},\preceq_{\mu},\ifu_{\mu}\>$ such that 
for every $i,j<\sigma$ 
the following holds:
\begin{enumerate}[(a)]
  \item $h_{{\nu},{\mu}}\restriction \aker=\operatorname{id},$
  \item $h_{{\nu},{\mu}}(\xn i)=\xm i,$
\item $\pib(\xn i)=\pib(\xn j)$  iff $\pib(\xm i)=\pib(\xm j),$ 
\item
$\pib(\xn i)=S$ iff $\pib(\xm i)=S,$
\item if $\{\xn i ,\xn j\}\in \br \aker;2;$ then $\xn i=\xm i$, $\xn j=\xm j$ and 
  $\ifu_{\nu}\{\xn i,\xn j\}=\ifu_{\mu}\{\xm i,\xm j\}$,
\item $\pi(\xn i)\in \oorb{\xn j}$ iff $\pi(\xm i)\in \oorb{\xm j}$ , 
\item $\pi(\xn i)\in \eorb{\xn j}$ iff $\pi(\xm i)\in \eorb{\xm j}$,  
\item $\pi(\xn i)\in \orb{\xn j}$ iff $\pi(\xm i)\in \orb{\xm j}$,
\item $\pi(x_{\nu,k})\in \dorb{x_{\nu,i}}{x_{\nu,j}}$  iff 
$\pi(x_{\mu,k})\in \dorb{x_{\mu,i}}{x_{\mu,j}}$.
\end{enumerate} 
\end{enumerate}
  Note that in order to obtain (C)(e) we use condition (P5) and the fact that 
  $|\dorb{x}{y}|\leq \kappa$ for all $x\ne y$. 
 
 Also, we may assume the following:

\begin{enumerate}[(A)]
\addtocounter{enumi}{3}\label{thin:partition}
\item There is a partition ${\sigma}=K\cup^* F\cup^* 
{D}\cup^* {M}$ such that for each ${\nu}<{\mu}<\oot$:
\begin{enumerate}[(a)]
\item $\forall i\in K$ $\xn i\in \aker$ and so $\xn i=\xm i$. $\aker=\{\xn
i:i\in K\}$. 
\item $\forall i\in F$ $\xn i\ne \xm i$ but $\pib(\xn
i)=\pib (\xm i)\ne S$.
\item  $\forall
i\in {D}$ $\xn i\notin \aker$, 
 $\pib (\xn i)=S$ and  ${\pi}(\xn i)\ne {\pi}(\xm i)$.
\item $\forall i\in {M}$ $\pib (\xn i)\ne S$ and $\pi (\xn i)\ne
\pi (\xm i)$.
\end{enumerate}
\item If $\pi ({\xn i})=\pi ({\xn j})$ then $\{i,j\}\in  \br K\cup F;2; \cup \br D\cup {M};2;$.

\end{enumerate}

It is well-known that if   $\sigma<\oo=\oo^{<\oo}$ then the
following partition relation holds:
\begin{equation}
\notag \oot\to (\oot, ({\omega})_{\sigma})^2.
\end{equation}
Hence we can assume:
\begin{enumerate}[(A)]
\addtocounter{enumi}{5}
\item \label{g} ${\pi}(\xn i)\le {\pi}(\xm i)$ 
for each   $i\in {\sigma}$ and ${\nu}<{\mu}<\oot$. 
\end{enumerate}

For $i \in {\sigma}$ let
\begin{equation}
 \notag
{\climit}_i=\left\{
\begin{array}{ll}
{\pi} (\xn i)&\text{if $i\in K\cup F$},\\
\sup\{{\pi} (\xn i):{\nu}<\oot \}&\text{if $i\in D\cup M$.}
\end{array}
\right.
\end{equation}

\begin{claim}
\label{Claim-2.5}  

 (a) If $i\in 
D\cup M$, then the sequence $\<\pi(\xn i):{\nu}<{\kappa}^+\>$ is strictly increasing, $cf(\delta_i) = \kappa^+$ and   $\mbox{
sup}(J(\delta_i))=\delta_i$.   Moreover 
for
every $\nu < \kappa^+$  we have  
$\pi(x_{\nu,i}) < \delta_i$ .

(b) If $\{i,j\}\in  [M]^2$ and $x_{\nu,i} \preceq_{\nu}
x_{\nu,j}$, then $x_{\nu,i} =
x_{\nu,j}$.

\end{claim}

\begin{proof}
If $i\in D\cup M$, then 
(F) and (D)(c-d)  imply that  the sequence $\{{\pi} (\xn i):{\nu}<\oot \}$
 is strictly increasing. Hence $cf(\delta_i) = \kappa^+$ 
and $\pi(\xn i)<{\delta}_i$
for $i\in D\cup M$. 

Thus 
Proposition \ref{Proposition-2.1}
implies $\mbox{
sup}(J(\delta_i))=\delta_i$.
Thus (a) holds.

(D)(d) and condition (P3) imply (b).

\end{proof}

 We put
\begin{equation*}
Z_0 = \{{\climit}_i:i\in {\sigma}\}.
\end{equation*}

Since $\pi''\aker=\{{\delta}_i:i\in K\}$ we have $\pi''\aker\subs Z_0$.
Then, we define $Z$ as the closure of $Z_0$ with respect
to $\mathbb I$:

\begin{equation*}
  Z = Z_0\cup \{I^+: I\in \mathbb I,I\cap Z_0\neq
  \emptyset\}.
  \end{equation*}
Observe that 
\begin{equation*}
 |Z|<{\kappa}.
\end{equation*}

\medskip

By Claim \ref{Claim-2.5}(a), 
the sequence $\<\pi(\xn i):{\nu}<{\kappa}^+\>$ is strictly increasing  
for $i\in D\cup M$. 
Since $|Z|<{\kappa}$, and     $|\orb{\xn k}|\le {\kappa} $ for $\xn k\in B_S\cap \aker$, we can assume that  
\begin{enumerate}[(A)]
\addtocounter{enumi}{6} 
 \item \label{omitostar}   
 $\pi(\xn i)\notin \orb{\xn k}$ for $\xn k\in B_S\cap \aker$ and $i\in D\cup M$.
\end{enumerate}

\medskip

Our aim is to prove that there are $\nu < \mu < \kappa^+$ such that the forcing 
conditions $r_{\nu}$ and  $r_{\mu}$ are compatible. However, since we are 
dealing with infinite forcing conditions, we will need to add new elements to 
$A_{\nu} \cup A_{\mu}$ in order to be able to define the infimum of pairs of 
elements $\{x,y\}$ where $x\in A_{\nu}\setminus A_{\mu}$  and $y \in 
A_{\mu}\setminus A_{\nu}$. The following definitions will be useful to 
provide the  room we need to insert  the required new elements.

Let
    $$\sigma_1 = \{i\in \sigma\setminus K: \cf(\delta_i)=\kappa
    \}$$
and
    $$\sigma_2 = \{i\in \sigma\setminus K: \cf(\delta_i)=\kappa^+
    \}.$$
    
\vspace{2mm}
    \noindent Assume that $i\in \sigma_1\cup \sigma_2$. Let
    $$\xi_i = \mbox{ min}\{\zeta \in \cf (\delta_i): \epsilon^{J(\delta_i)}_{\zeta}
    > \mbox{ sup}(\delta_i\cap Z)\}.$$
  Since $|Z|<{\kappa}\le cf({\delta}_i)$, the ordinal ${\xi}_i$ is defined  and
  ${\delta}_i>\epsilon^{J(\delta_i)}_{{\xi}_i}$.
    
    \noindent Then, if $i\in \sigma_1$ we put
    $$\underline{\gamma}(\delta_i) = \epsilon^{J(\delta_i)}_{\xi_i} \mbox{
    and } \gamma(\delta_i) = \delta_i,$$
    \noindent and if $i\in \sigma_2$ we put
    $$\underline{\gamma}(\delta_i) = \epsilon^{J(\delta_i)}_{\xi_i} \mbox{
    and } \gamma(\delta_i) = \epsilon^{J(\delta_i)}_{\xi_i + \kappa}.$$

For $i\in \sigma_2$,
since 
$\gamma(\delta_i)<\climit_i$ and $\climit_i=\lim\{{\pi} (\xn i):{\nu}<\oot \}$
by Claim \ref{Claim-2.5}(a)   for all $i\in D\cup M$,
we can assume that 
\begin{enumerate}[(A)]
\addtocounter{enumi}{7} 
\item \label{good1}\label{Claim-2.6}  
$  \pi(x_{\nu,i})\in
 J(\delta_i)\setminus \gamma(\delta_i)$, and so ${\pi} (\xn i)\notin Z$,
for all $i\in D\cup M$.
\end{enumerate}

We will use the following
 fundamental facts.

\begin{claim}
\label{Claim-2.10}  If $ \xp i
\prep \xp j$ then ${\climit}_i\le{\climit}_j$.
\end{claim}

\begin{proof} $\xp i
\prep \xp j$  implies $\pi(\xp i)\le
\pi( \xp j)$ by (P2). 
\end{proof}

\begin{claim}
\label{Claim-2.11}  Assume $i,j\in
{\sigma}$. If $\xp i\prep \xp j$ then either
${\climit}_i={\climit}_j$ or there is $a\in \aker\cap B_S$ with $\xp
i\prep a\prep \xp j $.
\end{claim}

\begin{proof}
 Assume that $i,j\not\in K$ and
$\delta_i\neq \delta_j$. By Claim \ref{Claim-2.10}, we have $\delta_i
<\delta_j$. Since $i\in F\cup M$ and  
$\xp i\prep \xp j$
imply  $\xp i = \xp j$
and so   $  \delta_i =\delta_j$, we
have that $i\in  D$, and so $\pi(\xn i)<\delta_i$, $\mbox{
cf}(\delta_i)=\kappa^+$ and $J(\delta_i)^+=\delta_i$  by 
Proposition \ref{Proposition-2.1} . 

Since $\delta_i <\delta_j$, 
we have $\delta_i<\gamma({\delta}_j)< \pi(\xp j)$ by (\ref{good1}), and so 
 $J(\delta_i)$ separates
$\xp i$ from $\xp j$. 
 By (P6), we infer
that there is an $a=\xp k\in \ap$ such that
$\pi(a)=\delta_i$ and $\xp i\prep a \prep \xp j$.

Since $\xn k\ne \xp j$, we have $\xn k\in B_S$, and so   $k\in K\cup D$.
But as $\pi(\xn k)= \delta_i\in Z$  we obtain $k\notin D$ by  (\ref{good1}), and so  $k\in K$, which implies     $a=\xn k\in \aker\cap B_S$.
\end{proof}

\begin{claim}
\label{Claim-2.12} 
If  $\xp i \in \aker\cap \block{S}$
and $ \xp j\in \ap $ are compatible but incomparable in $r_{\nu}$, then 
$\xp k=\ip \{\xp i,\xp j\}\in \aker\cap \block{S}$.
\end{claim}

\begin{proof}
First, (P2) implies $\xp k \in B_S$.

Since 
${\pi}(\xp k)={\pi}(\ip{\{\xp i,\xp j\}})
\in \dorb{\xp i}{\xp j}=\orb {\xp i}\cap \orb {\xp j} \subs 
\orb {\xp i}$ by (P5), and $\xp i \in \aker\cap \block{S}$, we have $k\notin D\cup M$ by (\ref{omitostar}).
Thus $k\in K$, and so  $\xp k\in \aker$.

Hence $\xp k=\ip \{\xp i,\xp j\}\in \aker\cap \block{S}$.
\end{proof}

\begin{claim}
\label{Claim-2.13} Assume that  $\xp i $
and $ \xp j$ are compatible but incomparable in $r_{\nu}$. Let
$\xp k=\ip\{\xn i,\xn j\}$. Then either $\xp k\in \aker$ or
${\climit}_i={\climit}_j={\climit}_k$.
\end{claim}

\begin{proof}
Assume $\xp k\not\in \aker$, i.e.  $k\not\in
K$. If 
${\climit}_k\ne {\climit}_i$, we infer that there is $b\in \aker\cap \block{S}$
with $\xp k\prep b \prep \xp i$ by Claim 
\ref{Claim-2.11}.
 So
$\xp k=\ip\{b,\xp j\}$ and thus $\xp k\in \aker$ by Claim \ref{Claim-2.12},
contradiction.

Thus ${\climit}_i={\climit}_k$, and similarly
${\climit}_j={\climit}_k$.
\end{proof}


\begin{definition}
\label{Definition 2.8} $\{r_{\nu},r_\mu\}$  is a  {\em
good pair} iff
for all $ \{i,j\}\in \br F;2;$ with  ${\climit}_i={\climit}_j$ and 
     $\cf ({\climit}_i)=\oot$ we have 
\begin{equation}\tag{$\blacktriangle$}\label{eq:good}
 \dorb{\xn i }{\xm j}\supset \overline{o}(\delta_i)\cap \gamma(\delta_i).
\end{equation}
\end{definition}

\begin{claim}\label{good-pair}
There are ${\nu}<{\mu}<{\kappa}^+$ such that the pair $\{r_{\nu},r_{\mu}\}$ is good. 
\end{claim}

\begin{proof}
Let $$\vartheta = \mbox{sup}\{\xi_\ell + \kappa : \ell\in  \sigma_2\cap F\}.$$
Since $\mc F$ is a $\kappa^+$-strongly unbounded function on $\lambda $
we can find 
${\nu}<{\mu}<{\kappa}^+$ such that for all 
$ \{i,j\}\in \br F;2;$ with  ${\climit}_i={\climit}_j$ and 
     $\cf {\climit}_i=\oot$ we have 
\begin{equation*}
 \mc F\{\rho({\xn i }),\rho({\xm j})\}\ge  \vartheta.
\end{equation*}
Hence \eqref{eq:good} holds.
 \end{proof}

To finish the proof of Lemma \ref{Lemma-2.3} we will show that 
\begin{equation}\tag{$\dag$}
\text{If $\{r_{\nu},r_{\mu}\}$ is a good pair, then 
$r_{\nu}$ and $r_{\mu}$ are compatible.} 
\end{equation}

So, assume that $\{r_{\nu},r_{\mu}\}$ is a good pair.

Write
$\climit_{\xp i}=\climit_{\xq i}=\climit_i$.

If $s=\xp i$ write $s\in K$ iff $i\in K$. Define  $s\in
F$, $s\in {M}$, $s\in {D}$ similarly.

In order to amalgamate conditions $r_{\nu}$ and $r_{\mu}$, we will use a
refinement of the notion of amalgamation given in \cite[Definition
2.4]{M}. 

 Let
$A'=\{\xp i:i\in F\cup D\cup M\}$.  For $x\in (\ap\setm \aq)\cup (\aq\setm \ap)$ define the {\em twin $x'$} of $x$ in a natural way:
$x'=h_{{\nu},{\mu}}(x)$ for $x\in \ap\setm \aq $, and 
$x'=h^{-1}_{{\nu},{\mu}}(x)$ for $x\in \aq\setm \ap$.

Let $\rank:\<A',\prep\restriction A'\>\to \theta$ be an
order-preserving injective function for some ordinal
$\theta<\oo$,
and  for $x\in A'$  let 
\begin{equation*}
 {\beta}_x=\eps{\delta_x}{\underline{\gamma}(\delta_x)+\rank(x)}.
\end{equation*}

Since
{$\cf(\gamma(\delta_x))= \oo$}
and $|A'|<{\kappa}$ we have
\begin{equation*}
{\beta}_x\in \bigl(\eorb{{\climit}_x}\cap
[\underline{\gamma}(\delta_x),\gamma(\delta_x))) \setm \sup\{
  {\beta}_{z}: \rank(z)<\rank( x)
\}.
\end{equation*}

For $x\in A'$ let
\begin{equation*}
y_x=
\<{\beta}_x,0\>,
\end{equation*}
and put
\begin{equation*}
Y=\{y_x:x\in A'\}.
\end{equation*}
So, for every $x\in A'$, $y_x\in B_S$ with $\pi(y_x)<\pi(x)$.

Define functions $g:Y\to \ap$ and $\gbar:Y\to \aq$ as follows: 
\begin{equation*}
\text{ $g(y_x)=x$ and $\gbar(y_x)=x'$},
\end{equation*}
where $x'$ is the ``twin'' of $x$ in $\aq$.

\vspace{2mm} Now, we are ready to start to define the common
extension $r=\<\ar,\prer,\ir\>$ of $r_{\nu}$ and $r_{\mu}$. First, we define
the universe $\ar$ as 
$$\ar = \ap\cup \aq\cup Y.$$

Clearly, ${\ar}$ satisfies (P1). Now, our purpose is to define
$\preceq$.

Extend the definition  of $g$ as follows: $g:\ar\to \ap$ is a
function,
\begin{displaymath}
g(x)=\left\{ 
\begin{array}{ll}
x&\text{if $x\in \ap$,}\\  
x'&\text{if $x\in \aq\setm \ap$,}\\  
s&\text{if $x=y_s$ for some $s\in A'$.}  
\end{array}
\right .  
\end{displaymath}

We introduce
two relations on $A_p\cup A_q\cup Y$ as follows:
\begin{equation}\notag
  \begin{array}{lcl}
\pre{1}&=&\{\<y,x\>\in Y\times \ar: g(y)\prep g(x)\},\medskip\\
\pre{2}&=&\{\<x,z\>\in  \ar\times \ar:  \exists a\in \aker\
g(x)\prep a \prep g(z)\}.
  \end{array}
\end{equation}
Then, we put
\begin{equation*}\tag{$\bigstar$}\label{preceq}
\prer=\preceq_{\nu}\cup \preceq_{\mu}\cup \pre{1}\cup \pre{2}.
\end{equation*}

The following claim is well-known and straightforward.
\begin{claim}
\label{Claim-2.18}  $\prepq=\prer\restriction (\ap\cup \aq)$ is the
partial order on $\ap\cup \aq$ generated by $\prep\cup \preq$.
\end{claim}

The following straightforward claim 
will be used several times in our arguments.
\begin{claim}\label{pr:xlez-gxegz}
If $x\prer z$ then $g(x)\prep g(z)$. 
\end{claim}

\begin{sublemma}\label{cl:prerprec}
\label{Lemma-2.19} $\prer$ is a partial
order on ${\ap}\cup {\aq}\cup Y$.
\end{sublemma}

\begin{proof}We should check that $\prep$ is transitive, because it is trivially reflexive and antisymmetric.

So
let $s\prer t\prer u$. We should show that $s\prer u$.

Since $x\prer z$ implies $g(x)\prep g(z)$, we have 
$g(s)\prep g(t)\prep g(u)$ and so 
\begin{equation}\tag{$\star$}\label{eq:gsgr}
 g(s)\prep g(u).
\end{equation}
If $\<s,u\>\in (Y\times {\ar})\cup ({\ap}\times {\ap})\cup ({\aq}\times {\aq})$,
then $\eqref{eq:gsgr}$ implies $s \pre1 u$ or $s \preceq_{\nu} u$  or $s\preceq_{\mu} u$, which implies   $s\prer u$   by  
(\ref{preceq}).

So we can assume that $s\in \ap$  (the case $s\in \aq$ is similar), and so 
$u\in Y$ or $u\in \aq$.  

\begin{case}
$u\in \aq$. 
\end{case}

If $t\in \ap\cup \aq$, then $s\prepq t\prepq u$, and so $s\prepq u$ by Claim 
\ref{Claim-2.18}. So $s\prer u$.

Assume that $t\in Y$. Then $s\pre{2}t $, and so 
 there is $a\in \aker$ such that $g(s)\prep a\prep g(t)$. Since $t\prer u$ implies   $ g(t)\prep g(u)$, 
we have $g(s)\prep a\prep g(u)$, and so $s\pre 2 u$. Thus $s\prer u$.

\begin{case}
$u\in Y$. 
\end{case}

If $t\in Y$, then $s{\pre 2}t$, and so  
 there is $a\in \aker$ such that $g(s)\prep a\prep g(t)$. Since $t\preceq u $ implies $g(t) \prep g(u)$,
we have  $g(s)\prep a\prep g(u)$, and so $s{\pre 2}u$. Thus $s\prer u$.

Assume that $t\in \ap\cup \aq$. Then $t {\pre 2}u$, and so there is 
$a\in \aker$ such that $g(t)\prep a \prep g(u)$.  
Then $g(s)\prep a\prep g(u)$, and so $s{\pre 2}u$. Thus $s\prer u$.
\end{proof}

So, by the previous Sublemma \ref{cl:prerprec} and by the construction, 
(P2) and (P3) hold for $\prer$.

Next
define the function $\ir: \br \ar;2;\to \ar\cup\{\undef\}$
as follows: 
\begin{equation*}
\ir\supset \ip\cup \iq, 
\end{equation*}
and for $\{s,t\}\in \br \ar;2;\setm (\br \ap;2;\cup \br \aq;2;)$
such that $s$ and $t$ are $\preceq$-compatible, put 
$\ir\{s,t\}=\ir\{s,y_s\}=\ir\{t,y_s\}=y_s$ if $s\in A'$
and $t=s'$, and otherwise
consider the element 
\begin{equation*}
 v=\ip\{g(s),g(t)\},
\end{equation*}
 and let
\begin{displaymath}
\ir\{s,t\}=\left\{ 
\begin{array}{ll}
v&\text{if $v\in \aker$,}\medskip\\  
y_{v}&\text{if $v\notin \aker$.}  
 \end{array}
\right .  
\end{displaymath}
Let $$\ifu\{s,t\}=undef$$ if $s$ and $t$ are not $\preceq$-compatible.

If $s$ and $t$ are compatible, then so are $g(s)$ and $g(t)$
because  $x\prer y$ implies $g(x)\prep g(y)$ by Claim \ref{pr:xlez-gxegz}.
Moreover $\ip\{s,t\}=\iq\{s,t\}$ for $\{s,t\}\in \br \aker;2;$
by condition (C)(e),
so the definition above is meaningful, and gives a function $\ir$.

\begin{claim}\label{cl:monsgs}
If $v\in \aker$ and  $s\in \ar$,  then   
$\pi(v)\in \orb{g(s)}$ iff $\pi(v)\in \orb s$. 
\end{claim}
\begin{proof}
If $s\in \ap\cup \aq$ then  $g(s)=s$ or $g(s)=s'$, and so  
$\pi(v)\in \orb{g(s)}$ iff  $\pi(v)\in \orb s$ by  (C)(b) and  (C)(h).

Consider now the case $s=y_x\in Y$. 
Then 
$\pi(s)\in E(J(\delta_{x}))\cap 
[\underline\gamma({\delta}_x),\gamma({\delta}_x))$, and so 
\begin{multline*}
 \orb{s}=\oorb{{\pi(s)}}=
\bigcup\{E(I): I\in \mbb I,   I^-<{\pi(s)}<I^+\}\cap {\pi(s)}=\\
\bigcup\{E(I):I\in \mbb I,  J(\delta_x)\subs I\}\cap {\pi(s)}.
\end{multline*}

We distinguish the following two cases.
\newcases
\begin{case}
$\pi(x)<\delta_x$. 
\end{case}

If $x\in \block S$ then   $\gamma(\delta_x)<\pi(x)<\delta_x$
by (\ref{good1}), 
and so 
\begin{equation*}
\orb{x}\cap \pi(s)=\oorb{\pi(x)}\cap \pi(s) = \bigcup \{E(I) : J(\delta_x)\subset I \} \cap 
\pi(s).
\end{equation*}

If  $x\notin \block S$  then   $x\in M$ and  
$\gamma(\delta_x)<\pi(x)<\delta_x$
by (\ref{good1}), and so 
\begin{multline*}
\orb{x}\cap \pi(s)=\eorb{\pi(x)}\cap\pi(s)=\\  (\bigcup\{E(I) : J(\delta_x)\subset I\} \cup 
E(J(\pi(x))))\cap  \pi(s) =\\\bigcup\{E(I) : J(\delta_x)\subset 
I\}\cap  \pi(s) = \orb{s}\cap \pi(s).
\end{multline*}

\begin{case}
$\pi(x)=\delta_x$. 
\end{case}
Then $x\in F$ and so
\begin{multline*}
 \orb{x}=\eorb{\pi(x)}=o(x)\cup (E(J({\delta_x}))\cap {\delta}_{x})=\\
\big(\bigcup\{E(I):  I^-<\pi(x)<I^+\}\cup E(J({\delta_x}))\big)\cap \pi(x)=\\
\bigcup\{E(I):  J(\delta_x)\subs I\}\cap \pi(x),
\end{multline*}
so $\orb{s}\cap \pi(s)=\orb{x}\cap \pi(s)$.

\medskip
So in both cases $\orb{s}\cap \pi(s)=\orb{x}\cap \pi(s)$. Also, note that
as $v\in \aker$, we have that 
$\pi(v)\notin (\underline{{\gamma}}({\delta}_x), {\delta}_x)$,
and hence if $v\in \orb{g(s)}$ then $\pi(v)<\pi(s)$. So,
$\pi(v)\in \orb x =\orb{g(s)}$ iff $\pi(v)\in   \orb s$.
\end{proof}

\begin{claim}\label{cl:vP5}
 If $\{s,t\}\in \br \ar;2;$, $v\in \aker$  and 
$\pi(v)\in \dorb{g(s)}{g(t)}$ then $\pi(v)\in \dorb st$.
\end{claim}

\begin{proof}We should distinguish two cases.
\newcases 
\begin{case}
 $\dorb{s}{t}=\orb s\cap \orb{t}$. 
\end{case}

Since
$\dorb{g(s)}{g(t)}= \orb{g(s)}\cap \orb{g(t)}$
we have $\pi(v)\in\orb{g(s)}\cap \orb{g(t)}$. 
Since  $\pi(v)\in \orb{g(s)}$ implies  $\pi(v)\in \orb s$
and $\pi(v)\in \orb{g(t)}$ implies  $\pi(v)\in \orb t$
by Claim \ref{cl:monsgs}, we have $\pi(v)\in \dorb st$.

\begin{case}\label{ff}
$\dorb{s}{t}= o(s)\cup\big\{\eps {\pi(s)}\zeta:\zeta <F\{ \rho(s), \rho(t)\}\big\}$.
\end{case}

We can assume that $s\in \ap\setm \aq$ and $t\in \aq\setm \ap$.
Then  $s,t\in F$, $\delta'=\delta_s=\delta_t$ has cofinality  ${\kappa}^+$ 
and so  $g(s), g(t)\in F$ and $\delta_{g(s)}=\delta_{g(t)}=\delta'$
and so \begin{equation*}\tag{$\vartriangle$}\label{vinorb}
\pi(v)\in \dorb{g(s)}{g(t)}= \oorb{\delta'}\cup\big\{\eps {\delta'}\zeta:
\zeta <\mathcal F\{\rho(g(s)), \rho(g(t))\}\big\}.
\end{equation*}
Since $\pi''\aker\cap (\gamma(\delta'),\delta')=\empt$,
(\ref{vinorb}) implies  $$\pi(v)\in \eorb{\delta'}\cap \gamma(\delta').$$
But, by \eqref{eq:good} 
\begin{equation*}
\eorb{\delta'}\cap \gamma(\delta')\subs \dorb{s}{t}\cap \gamma(\delta'),  
\end{equation*}
and so $\pi(v)\in \dorb{s }{t}$.
\end{proof}

\begin{sublemma}
\label{Lemma 2.21}  
 $\<\ar,\prer,\ir\>$ satisfies
(P4) and (P5).
\end{sublemma}

\newcases
\begin{proof}
Let $\{s,t\}\in \br {\ar};2;$ be a pair of $\prer$-incomparable and 
$\prer$-compatible
elements.
\begin{case}
$\{s,t\}\in \br \ap;2;$. (The case $\{s,t\}\in \br \aq;2;$ is similar) 
\end{case}

Since $\prep\subs \prer$, we have $\ip\{s,t\}\prer s,t$, 
so to check (P4) we should show that $x\prer s,t$ implies 
$x\prer \ip\{s,t\}$.
We can assume that $x\notin \ap.$

If $x\in Y$, then $x\pre1 s$ and 
$x\pre1 t$, i.e. $g(x)\prep g(s), g(t)$ and so 
$g(x)\prep \ip\{g(s),g(t)\}=\ip\{s,t\}=g(\ip\{s,t\})$, and so $x\pre1\ip\{s,t\}$.
Thus $x\prer \ip\{s,t\}$.

If $x\in \aq\setm \ap$, then $x\pre{2} s$ and 
$x\pre{2} t$, i.e. $g(x)\prep a\prep  g(s)$ and $g(x)\prep b\prep g(t)$
for some $a,b\in \aker$. Then $c=\ip\{a,b\}\in \aker$, and so 
$g(x) \prep c\prep \ip\{g(s),g(t)\}=\ip\{s,t\}=g(\ip\{s,t\})$, and so $x\pre{2}\ip\{s,t\}$.
Thus $x\prer \ip\{s,t\}$.

Finally 
(P5) holds in Case 1 because $r_{\nu}$ satisfies (P5).

\begin{case}
$\{s,t\}\notin \br \ap;2;\cup \br \aq;2;$. 
\end{case}

To check (P4) we should prove that 
 $\ir\{s,t\}$ is the greatest common lower bound of 
$s$ and $t$ in $\<\ar,\prer\>$.

Assume first that $s$ and $t$ are not twins. Note that by Claim \ref{pr:xlez-gxegz},
$g(s)$ and $g(t)$ are $\prep$-compatible.
Write $v=\ip\{g(s), g(t)\}$.

\begin{scase}
$v\in \aker$, and so $\ir\{s,t\}=v$.  
\end{scase}

Since $v=g(v)\prep g(s)$ and $v\in \aker$, we have $v\pre{2} s$.
Similarly $v\pre{2} t$. Thus $v$ is a common lower bound of $s$ and $t$.

To check that $v$ is
the greatest lower bound of $s,t$ in $\<{\ar},\prer\>$
let $w\in \ar$, $w\prer s,t$. Then 
$g(w)\prep g(s),g(t)$. 
Thus $g(w) \prep \ip\{g(s), g(t)\}=v$.

Since $v\in \aker$, 
$g(w)\prep v$ implies $w\pre{2} v$. Thus $w\prer v$.  Thus (P4) holds.

To check (P5)
observe that $g(s)$ and $g(t)$ are incomparable in $\ap$.
Indeed, $g(s)\prep g(t)$ implies $v=g(s)\in \aker$ and so
$g(s)\prep g(t)$ implies $s\pre{2} t $, which contradicts our assumption
that $s$ and $t$ are $\prer$-incomparable.

Thus,  by applying (P5) in $r_{\nu}$,
\begin{equation*}
{\pi}(v)\in \dorb{g(s)}{g(t)}. 
\end{equation*}

Thus $\pi(v)\in \dorb{s}{t}$ by Claim \ref{cl:vP5}, and so (P5) 
holds.

\begin{scase}
$v\notin \aker$, and so $\ir\{s,t\}=y_v$.  
\end{scase}

If $g(s)$ and $g(t)$ are 
$\prep$-comparable then $\delta_{g(s)} = \delta_{g(t)}$, because otherwise 
we would infer from Claim \ref{Claim-2.11} that 
$s,t$ are $\prer$-comparable, which is 
impossible.

Now assume that 
 $g(s)$ and $g(t)$ are 
$\prep$-incomparable.

If $\delta_v<\delta_{g(s)}$, then there is $a\in \aker\cap B_S$
with $v\prep a\prep g(s)$ by Claim \ref{Claim-2.11}. Thus $v=\ip\{a, g(t)\}$ and so $v\in \aker$
by Claim \ref{Claim-2.12}.
Thus $\delta_v=\delta_{g(s)}$, and similarly $\delta_v=\delta_{g(t)}$.

Hence we have $$\delta_{g(s)}=\delta_{g(t)}=\delta_v$$ in both cases. 
Thus $$\pi(y_v)\in   E(J({\delta_v})) 
\cap [\underline {\gamma}({\delta_v}),{\gamma}({\delta_v})).$$
If $s,t \in F$ and $\cf({\delta}_v)={\kappa}^+$, by condition \eqref{eq:good}, we deduce that 
$E(J({\delta}_v))\cap {\gamma}({\delta}_v)\subs f\{s,t\}$, and so as 
$\pi(y_v)<{\gamma}({\delta}_v)$, we have $\pi(y_v)\in f\{s,t\}$. Otherwise, 
\begin{equation*}
E(J({\delta_v}))\cap \min(\pi(s),\pi(t))\subs \dorb{s}{t}. 
\end{equation*}
Then as $v=\ip\{g(s),g(t)\}$, we have $\pi(v)<\pi(g(s)), \pi(g(t))$, hence
$\pi(y_v)<\pi(s), \pi(t)$ and thus  $\pi(y_v)\in \dorb{s}{t}$.

Thus (P5) holds.

To check (P4)
first we show that $y_v\prer s,t$.  Indeed $g(v)\prep g(s)$ implies
$y_v\pre1 s$. We obtain $y_v\pre1 t$ similarly.

Let $w\prer s,t $.

Assume first that 
${\delta}_{g(w)}<{\delta}_v$.  
Since $w\prer s,t$ we have $g(w)\prep g(s), g(t)$  by Claim \ref{pr:xlez-gxegz}
and hence $g(w)\prep \ip\{g(s),g(t)\}=v.$   By Claim \ref{Claim-2.11} there is $a\in \aker$
such that $g(w)\prep a \prep v$. 
 Thus $w\pre{2} y_v$.

Assume now that 
${\delta}_{g(w)}={\delta}_v$. 

Then, we have that $w\in Y$. To check this fact, 
assume on the contrary that $w\in {\ap} \cup {\aq}$. So, we have $\delta_w = 
\delta_{g(w)} = \delta_v = \delta_{g(s)} = \delta_{g(t)}$.  
Note that if $s\in Y$, then 
$\pi(s)\in [\underline{{\gamma}}({\delta}_w),{\gamma}({\delta}_w))$, which contradicts the 
assumption that  $w\prer s$. So $s\notin Y$, and analogously $t\notin Y$. 
Assume that $w\in {\ap}$. As $w\prer s,t$ and 
$[\gamma(\delta_v),J(\delta_v)^+)\cap Z = \empt$, it follows that $s,t \in {\ap}$, 
which was excluded. Analogously, $w\in {\aq}$ implies $s,t \in {\aq}$. 

Therefore, $w 
= y_z$ for some $z\in A'$.
Then $z\prep g(s)$ and $z\prep g(t)$,
and so $z\prep \ip\{g(s), g(t)\}=v$.
Thus $y_z\pre1 y_v$.

\medskip
If $s$ and $t$ are twins, then $s\in A'$ implies that $\ir\{s,t\}=y_s$ and we can proceed as above in Case 2.2.

So we proved Sublemma \ref{Lemma 2.21}.
\end{proof}

\begin{sublemma}
\label{Lemma 2.22}  $\<{\ar},\prer,\ir\>$
satisfies $(P6)$.
\end{sublemma}

\begin{proof}
Assume that $\{s,t\}\in \br \ar;2;$, $s\prer t$ and $\Lambda$
separates ${s}$ from ${t}$, i.e,
\begin{equation*}
 \Lambda^-<\pi(s)<\Lambda^+<\pi(t).
\end{equation*}

We should find $v\in {\ar}$
such that $s\;\prer\; v\;\prer\; t$ and ${\pi}(v)=\Lambda^+$. 

Note that since $s\prer t$, we have $\delta_{g(s)}\leq
\delta_{g(t)}$ by Claim \ref{Claim-2.10}.

We can assume that 
  $\{s,t\}\notin \br \ap;2;\cup \br \aq;2;$
because $r_{\nu}$ and $r_{\mu}$ satisfy (P6).

\newcases
\begin{case}
${\delta}_{g(s)}<{\delta}_{g(t)}$.  
\end{case}

As $g(s)\prep g(t)$,  there is 
 $a\in \aker\cap B_S$ with $g(s)\prep a \prep g(t)$
 by Claim \ref{Claim-2.11}. 

\begin{scase}
${\pi}(a)\in \Lambda$.  
\end{scase}

Thus $\Lambda$ separates $a$ from  $g(t)$.

Applying (P6)  in $r_{\nu}$  for $a$ and  $g(t)$ and $\Lambda$ we obtain $b\in \ap$  
 such that $a\prep b\prep g(t)$   and $
{\pi}(b)={\Lambda}^+$. 

       Note that as ${\pi}(a)\in \Lambda, a\in \aker$ and ${\pi}(b)
       =\Lambda^+$, we have that ${\pi}(b)\in Z$.
Thus $b\in \aker$ by (\ref{good1}).

Thus $g(s)\prep b\prep g(t)$ implies $s\pre2b\pre{2} t$, and so $s\preceq b \preceq t$.

\begin{scase}
${\pi}(a)\notin \Lambda$.  
\end{scase}

If $\Lambda^+=\pi(a)$, then we are done because
$g(s)\prep a \prep g(t)$ implies $s\prer a \prer t$.

So we can assume that $\Lambda^+<\pi(a)$. 

Since $r_{\nu}$ and $r_{\mu}$ satisfy $(P6)$ and $\Lambda$ separates $s$ from $a$, 
we can assume that $s\notin \ap\cup \aq$.

Hence $s=y_{g(s)}$ and 
$\Lambda$ separates $g(s)$ from $a$ because 
$\pi(s)\in J(\delta_{g(s)})\subs \Lambda$. 
(If $\Lambda\subsetneq J(\delta_{g(s)})$, then $\Lambda^-<\pi(s)<\Lambda^+$
is not possible.)

Thus there is $b\in \ap$ such that $g(s)\prep b\prep a$ and $\pi(b)=\Lambda^+$.

Since $\delta_{g(s)}\in Z_0$, we have $\pi(b)\in Z$, and so $b\in \aker$
by (\ref{good1}).

Thus $s=y_{g(s)}\pre1 b \pre{2} t$, and so $s\prer b\prer t$.

\begin{case}
${\delta}_{g(s)}={\delta}_{g(t)}$.  
\end{case}
 
We will see that this case is not possible. 
 
\begin{scase}
$s\in \ap$.  
\end{scase}

As $s\prer t$  and $[\gamma(\delta_s),J(\delta_s)^+)\cap Z=\empt$
we have that $t\notin {\aq}$.

Since $s\in {\ap}$, $s\prer t$ and $\delta_s = \delta_{g(t)}$
we have $t\notin Y$, and so $t\in \ap$, 
 which was excluded.

     By means of a similar argument, we can show that $s\in {\aq}$
     is also impossible.

\begin{scase}
$s=y_{g(s)}$.  
\end{scase}

Then $\pi(s)\in E(J(\delta_{g(s)}))$ and so 
$\Lambda^-<\pi(s)<\Lambda^+$ implies 
$J(\delta_{g(s)})\subs \Lambda$.
But then  $\pi(t)\le \Lambda^+$, so     
$\Lambda$ can not separate $s$ from $t$.

Thus (P6) holds.

Thus we proved Sublemma \ref{Lemma 2.22}.
\end{proof}

Thus we proved that $r$ is a common extension of $r_{\nu}$ and $r_{\mu}$.

This completes the proof of Lemma \ref{Lemma-2.3}, i.e. $\pcal$ satisfies
${\kappa}^+$-c.c.
\end{proof}

\end{document}